\documentclass{amsart}
\usepackage{hyperref}
\usepackage[T1]{fontenc}
\usepackage{times}
\usepackage{amssymb,verbatim}
\theoremstyle{plain}
\usepackage[T1]{fontenc}
\usepackage[utf8]{inputenc}
\usepackage[english]{babel}
\usepackage{amssymb}
\usepackage{amsthm}
\usepackage{graphics}
\usepackage{amsmath}
\usepackage{amstext}
\usepackage{multirow}
\usepackage{graphicx}
\usepackage{amsmath}
\usepackage{amssymb}
\usepackage{amsthm}
\usepackage{amstext}
\usepackage{mathrsfs}
\usepackage{amscd}
\usepackage{mathtools}

\usepackage{tikz}
\usetikzlibrary{cd}
\usetikzlibrary{matrix}

\newtheorem{thm}{Theorem}[section]
\newtheorem{dfn}[thm]{Definition}

\newtheorem{prop}[thm]{Proposition}

\newtheorem{conjecture}[thm]{Conjecture}
\newtheorem{problem}[thm]{Problem}

\newtheorem{THM}{Theorem}

\newcommand{\mc}{\mathcal}

\newcommand{\F}{\mc F}
\newcommand{\G}{\mc G}

\DeclareMathOperator{\sing}{sing}

\DeclareMathOperator{\Aut}{Aut}
\DeclareMathOperator{\Pic}{Pic}

\DeclareMathOperator{\Hilb}{Hilb}

\numberwithin{equation}{section}

\usepackage{mathtools}
\usepackage{ifmtarg}

\addtocounter{section}{0}             
\numberwithin{equation}{section}       

\title[Algebraic separatrices for  non-dicritical foliations]{Algebraic separatrices for  non-dicritical foliations on projective spaces of dimension at least four}

\author[J. V. Pereira]{Jorge  Vit\'orio Pereira}
\address{IMPA, Estrada Dona Castorina, 110, Horto, Rio de Janeiro,
Brasil}
\email{jvp@impa.br}


\makeatletter

\AtBeginDocument{%
  \@ifpackageloaded{refcheck}
  {%
    \@ifundefined{hyperref}{}{%
      \let\T@ref@orig\T@ref%
      \def\T@ref#1{\T@ref@orig{#1}\wrtusdrf{#1}}%
      \let\@refstar@orig\@refstar%
      \def\@refstar#1{\@refstar@orig{#1}\wrtusdrf{#1}}
      \DeclareRobustCommand\ref{\@ifstar\@refstar\T@ref}%
    }%
  }{}%
}

\makeatother

\begin{document}
\

\begin{abstract}
Non-dicritical codimension one foliations on projective spaces of dimension four or higher always have an invariant algebraic hypersurface. The proof relies on a strengthening of a result by Rossi on the algebraization/continuation of analytic subvarieties of   projective spaces.
\end{abstract}

\dedicatory{To  Felipe Cano
on the occasion of his 60th birthday}

\maketitle

\setcounter{tocdepth}{1}


\section{Introduction}

\subsection{Motivation} This paper draws motivation from  a conjecture proposed by Brunella concerning the structure of codimension one foliations on projective spaces.

\begin{conjecture}\label{Conj:Brunella}
Every codimension one foliation on $\mathbb P^n$, $n \ge 3$, either admits an invariant algebraic hypersurface or is everywhere tangent to a foliation by codimension two algebraic subvarieties.
\end{conjecture}

We focus our attention on the class of codimension one foliations with non-dicritical singularities. Roughly speaking, these are foliations for which composition of blow-ups with centers contained in the singular set of the foliation will have invariant exceptional divisors, see \S \ref{S:dicritical}. It can be verified that non-dicritical foliations on $\mathbb P^n$ cannot be tangent to a one-dimensional foliation by algebraic leaves. Therefore, a positive  answer to Conjecture \ref{Conj:Brunella}, would imply that non-dicritical foliations on $\mathbb P^n$, $n \ge 3$, have at least one  algebraic leaf.

\subsection{Existence of algebraic separatrices} Our first main result confirms Conjecture
\ref{Conj:Brunella} for the class of non-dicritical foliations on  $\mathbb P^n$, $n \ge 4$.

\begin{THM}\label{THM:Algebraic}
Let $\F$ be a codimension one foliation on $\mathbb P^n$, $n \ge 4$. If $\F$
is non-dicritical then $\F$ leaves invariant an algebraic hypersurface.
\end{THM}

Unfortunately, our proof of Theorem \ref{THM:Algebraic} does not generalize easily to dimension three.
Nevertheless, our arguments still guarantee the existence of algebraic separatrices in $\mathbb P^3$
under some  conditions on the codimension two components of the singular set of the foliation as explained in Section \ref{S:dim3}.

\subsection{Characterization of non-dicritical logarithmic foliations}
In \cite{MR2324555}, Cerveau, Lins Neto et al.,  proposed a stronger version  of Conjecture \ref{Conj:Brunella}.

\begin{conjecture}\label{Conj:Crocodile}
Every codimension one foliation on a compact complex manifold either is transversely projective  or is everywhere tangent to a foliation by codimension two compact subvarieties.
\end{conjecture}

Our second main result confirms Conjecture \ref{Conj:Crocodile} for  non-dicritical foliations on  $\mathbb P^n$, $n \ge 4$, with general two dimensional section free from  saddle-nodes.

\begin{THM}\label{THM:Logarithmic}
Let $\F$ be a codimension one foliation on $\mathbb P^n$, $n \ge 4$.
If $\F$ is non-dicritical and the restriction of $\F$ to a general $\mathbb P^2$
does not have saddle nodes in its resolution of singularities then $\F$ is defined
by a closed logarithmic $1$-form.
\end{THM}

It is interesting to compare Theorem \ref{THM:Logarithmic} with the main result of
\cite{MR3249047}. As in the case of Theorem \ref{THM:Algebraic} we can, under more restrictive assumptions on the codimension two  singularities of $\F$, formulate a version of Theorem \ref{THM:Logarithmic} valid for foliations on  $\mathbb P^3$, see  Section \ref{S:dim3}.

\subsection{Algebraization of analytic subvarieties} The main technical tool used in the proofs of Theorems \ref{THM:Algebraic} and \ref{THM:Logarithmic} is a strengthening of a classical result by Rossi \cite{MR0244516} concerning the algebraization  of germs of analytic subvarieties of projective varieties. Although not standard, we will use the terminology local subvarieties for subvarieties of Euclidean open subsets of projective spaces, in order to emphasize that they are not a priori  globally defined.

\begin{THM}\label{THM:Rossi}
Let $X$ be an irreducible subvariety of $\mathbb P^n$. Let $U$ be an Euclidean neighborhood of $X$ and let $V$ be a local irreducible subvariety of $U$. If $\dim V + \dim X > n$ and $X \cap V \neq \emptyset$
then there exists a subvariety $\overline V$ of $\mathbb P^n$ such that $\dim \overline V = \dim V$ and $\overline V \cap U \supseteq V$.
\end{THM}

Rossi, in his original statement, made the assumption  that $X \cap V$ have the expected dimension $\dim X + \dim V - \dim X \cap V$. Here we make no assumption on the dimension of $X \cap V$. The original proof is analytic in nature, and relies on an ingenious  use of Hartog's extension theorem. Our proof  is more algebraic, and explores properties of the Hilbert scheme of $\mathbb P^n$.

\subsection{Structure of the paper}
Section \ref{S:Cano} reviews the definition of non-dicritical singularities and the results on the existence of separatrices for  codimension one foliations in the local and in the semi-local setting. Section \ref{S:Rossi} is devoted to the proof of Theorem \ref{THM:Rossi}. Finally, Theorems \ref{THM:Algebraic} and \ref{THM:Logarithmic} are proved in Section \ref{S:proofs}.

\subsection{Acknowledgments} J. V. Pereira thanks Dominique Cerveau and Stefan Kebekus for their remarks about this work and acknowledges the support of  CNPq, Faperj,  and Freiburg Institute for Advanced Studies (FRIAS). The research leading to these results has received funding from the People Programme (Marie Curie Actions) of the European Union's Seventh Programme (FP7/2007-2013) under REA grant agreement n\textsuperscript{\underline{\scriptsize o}} 609305.

\section{Existence of separatrices for germs of non-dicritical foliations}\label{S:Cano}
This section briefly reviews the known results concerning the existence of separatrices for codimension one foliations on germs of smooth complex manifolds.

Let $\omega$ be a  germ of integrable differential $1$-form on $(\mathbb C^n,0)$ and consider the foliation $\mathcal F$ defined by it. As usual we will assume that $\omega$ has singular set of codimension at least two. A germ of hypersurface $H$ through $0$ is a separatrix for $\F$ if for every smooth point $p$ of $H$, the tangent space of $H$ at $p$ is contained in the kernel of $\omega(p)$.

\subsection{Camacho--Sad} For foliations on $(\mathbb C^2,0)$ there always exists at least one separatrix for $\F$. This was first established by Camacho and Sad in \cite{MR657239}.

\begin{thm}\label{T:CamachoSad}
Let $\mathcal F$ be a germ of foliation on $(\mathbb C^2,0)$. Then there exists a germ of separatrix through $0$.
\end{thm}

The result was later generalized by Camacho to singular surfaces with contractible resolution graph in \cite{MR963011}. Since then,  a number of alternative proofs came to light, see for instance \cite{MR1389507, MR1742334, MR1754036, MR2747734, phdEdileno}.

\subsection{Dicritical foliations}\label{S:dicritical}
 Theorem \ref{T:CamachoSad} does not generalize to codimension one foliations on higher dimensional manifolds without further hypothesis. Even before the appearance of \cite{MR657239}, there was available in the literature an example of a  foliation on $\mathbb C^3$ (global and homogeneous) without any germ of separatrix at $0$.  For any $m\ge 2$, the foliation $\F_m$ defined by the
$1$-form
\[
   \omega_m = (x^{m}z - y^{m+1})dx + (y^{m}x - z^{m+1}) dy + (z^{m}y - x^{m+1})dz
\]
does not have a separatrix at the origin, see \cite[Chapter 4]{MR537038}.

\begin{dfn}
A foliation $\F$ on $(\mathbb C^n,0)$ is dicritical if there exists a finite sequence of blow ups
\[
X_0  = (\mathbb C^n,0) \leftarrow X_1 \leftarrow \cdots \leftarrow X_N
\]
with smooth centers everywhere tangent to the strict transform $\F_i$ of the foliation $\F_0 = \F$, such the exceptional divisor of the last morphism is not invariant by $\F_N$.
\end{dfn}

The strict transforms of Jouanolou's foliations $\F_m$ under the  blow-up at the origin of $\mathbb (\mathbb C^3,0)$  give rise to  foliations which do not leave the exceptional divisor invariant. All Jouanolou's foliations are dicritical foliations.

The definition above appears in \cite[Section 2.1]{MR1162557} as the first of five equivalent definitions for dicritical foliations. The last definition (loc. cit.) states that a foliation $\F$ is dicritical if and only if there exists an irreducible surface $Z \subset (\mathbb C^n,0)$ such that restriction of $\F$ to $Z$
contains infinitely many distinct separatrices. For more information about dicritical foliations the reader can consult \cite{MR1010778, MR980953}.

\subsection{Cano--Cerveau and Cano--Mattei} It turns out that dicriticalness is the only obstruction for the existence of separatrices of codimension one foliations on smooth manifolds.

\begin{thm}\label{T:Cano}
Let $\F$ be a germ of codimension one foliation on $(\mathbb C^n,0)$, $n\ge 3$. If $\F$ is non-dicritical then there exists a germ of  invariant hypersurface through $0$.
\end{thm}

Theorem \ref{T:Cano}  is due to Cano and Cerveau \cite{MR1179013} when $n=3$ and to Cano and Mattei \cite{MR1162557} when $n>3$. Both proofs rely on reduction of singularities for non-dicritical foliations
established in full generality in dimension three \cite{MR1179013} and generically in dimension greather than three \cite{MR1162557}.

\subsection{Semi-local separatrices}
In dimension three, Cano and Cerveau prove a reduction of singularities theorem for non-dicritical foliations and then establish a bijection between connected components of the set of singular points of the resulting foliation which are not contained in the singular set of the exceptional divisor,  and separatrices (formal or convergent) for the original foliation on $(\mathbb C^3,0)$, see \cite[Theorem 2.1]{MR1179013}. In particular they show that any germ of curve in $(\mathbb C^3,0)$ tangent everywhere tangent to a foliation
but not contained in the singular set, is contained in a unique separatrix.  The existence of such germs of curves follows  from Theorem \ref{T:CamachoSad} applied to a sufficiently general hyperplane sections.

In dimension strictly greater than three, Cano and Mattei show in the proof of  \cite[Theorem 5 ]{MR1162557} that the separatrices of the restriction of $\mathcal F$ to a sufficiently general $3$-dimensional germ of manifold through the origin of $(\mathbb C^n,0)$ can be uniquely extended to germs of separatrices of the foliation on $(\mathbb C^n,0)$.

Putting together  \cite[Theorem 2.1]{MR1179013} and the proof of Theorem \cite[Theorem 5]{MR1162557}, one obtains the following semi-local version of Theorem \ref{T:Cano}.

\begin{thm}\label{T:semilocal}
Let $\F$ be a codimension one foliation on a complex manifold $X$. Let $S \subset \mathrm{sing}(\F)$ be an irreducible component of the singular set of $\mathcal F$; let $p \in S$ be a sufficiently general point of $S$; and let $\gamma$ be a germ of irreducible curve at $p$ not contained in $\sing(\F)$ but everywhere tangent to $\F$. If $\mathcal F$ is non-dicritical along $S$ then there exists an open Euclidean neighborhood $U$ of $S$ and a local $\F$-invariant hypersurface $V\subset U$ containing both $\gamma$ and  $S$.
\end{thm}

\section{Continuation of subvarieties}\label{S:Rossi}

\subsection{Rossi's Theorem}
We recall below the main result of \cite{MR0244516}. Its proof is analytic in nature and relies on an ingenious application of Hartog's Theorem.

\begin{thm}
Let $X$ be an irreducible subvariety of $\mathbb P^n$. Let $U\subset \mathbb P^n$ be an Euclidean neighborhood of $X$ and let $V$ be a local irreducible subvariety of $U$. If
\begin{enumerate}
\item $\dim V + \dim X > n$, and
\item every branch of  $V\cap X$ has dimension $\dim V + \dim X - n$
\end{enumerate}
then there exists a subvariety $\overline V$ of $\mathbb P^n$ such that $\overline V \cap U = V$.
\end{thm}

\subsection{A strengthening of  Rossi's Theorem}
At the introduction of  \cite{MR0244516}, Rossi remarks that he does not know
if condition (ii) is really necessary. A variant of the argument used in \cite[Proposition 6.6]{2011arXiv1107.1538L} shows that condition (ii) is indeed superfluous.

\begin{thm}[Theorem \ref{THM:Rossi} of Introduction]\label{T:Rossi}
Let $X$ be an irreducible subvariety of $\mathbb P^n$. Let $U$ be an Euclidean neighborhood of $X$ and let $V$ be a local irreducible subvariety of $U$. If $\dim V + \dim X > n$ and $X \cap V \neq \emptyset$ then there exists a subvariety $\overline V$ of $\mathbb P^n$ such that $\dim \overline V = \dim V$ and $\overline V \cap U \supseteq V$.
\end{thm}
\begin{proof}
Fix once and for all a closed point  $p \in X\cap V$. For any irreducible component $\Sigma$ of the Hilbert scheme of $\mathbb P^n$, let $\Sigma(V,p)$ be the
subset corresponding to subschemes containing $p$ and with formal completion at $p$
contained in $V$. By definition, $\Sigma(V,p)$ is the  intersection of the
Zariski closed subsets  $\Sigma_k (V,p) \subset \Hilb(\mathbb P^n)$ corresponding to subschemes containin $p$ and with $k$-th infinitesimal neighborhood at $p$ contained at the $k$-th infinitesimal neighborhood of $V$ at $p$. It follows that $\Sigma(V,p) \subset \Hilb(\mathbb P^n)$ is a Zariski closed subset.

Let  $\mathcal U_{\Sigma} \subset \Sigma \times \mathbb P^n \to \Sigma$ be the universal family of subschemes parametrized by $\Sigma$. We will denote by $\mathcal U_{\Sigma(V,p)}$  the restriction of the universal family to $\Sigma(V,p)$ and by $q : \mathcal U_{\Sigma(V,p)} \to \mathbb P^n$  the natural projection to $\mathbb P^n$. The proof will go by showing that $\overline V$ can be chosen to be equal to  $q(\mathcal U_{\Sigma(V,p)})$ for a suitably irreducible component $\Sigma$ of the Hilbert scheme of $\mathbb P^n$.

Let us fix a metric on $\mathbb P^n$. The subvariety $X$ is a compact subset of the open set $U$. As such, it rests at a positive distance $c >0$
from the boundary of $U$, i.e. $d(X, \partial U)= c> 0$. Therefore, there exists an open neighborhood $W \subset \Aut(\mathbb P^n)$ of the identity such that $d(\varphi^*X, \partial U) \ge c/2$ for any $\varphi \in W$. Notice that every irreducible component of $\varphi^*X \cap V$ is a projective subvariety of dimension at least $\dim X + \dim V - n> 0$ contained in $V$.

Let $W_p \subset W$ be the subset consisting of automorphisms which fix $p$. Let $\varphi_0 \in W_p$ be an automorphism for which there exists an irreducible component $E$ of $\varphi_0^* X \cap V$ containing $p$ and with minimal dimension among all irreducible components of $\varphi^* X \cap V$ containing $p$ for $\varphi \in W_p$.  If $q \in E$ is a point different from $p$ and not contained in any other irreducible  component of $\varphi_0^* X \cap V$, then varying $\varphi \in W_p$ and considering
the image $\varphi ( \varphi_0^{-1}(q))$  we obtain a full neighborhood $N_q$ of $q$ inside $V$. By construction this neighborhood is filled up by irreducible projective subvarieties containing $p$ and contained in $V$. We deduce the existence of an  irreducible component $\Sigma \subset \Hilb(\mathbb P^n)$  with  general element in $\Sigma(V,p)$ corresponding to an irreducible subvariety of $V$ and such that the   morphism
$q : \mathcal U_{\Sigma(V,p)} \to \mathbb P^n$ maps an analytic neighborhood of $\Sigma \times \{ p\} \subset \mathcal U_{\Sigma(V,p)}$
inside $V$. Moreover, the image of such analytic neighborhood also contains $N_q$. It follows that
$q(\mathcal U_{\Sigma(V,p)})$ is the sought algebraization of $V$.
\end{proof}

It was pointed out by Kebekus that the result above is probably not  formulated in its most general/natural form. The use of the generic $2$-transitiveness of the automorphism group of $\mathbb P^n$ should  be replaced by an abundance of deformations of $X$ inside the ambient manifold. We do not pursue this line of reasoning here. Anyway, we do believe that a better understanding of the mechanisms leading to the validity of the result should be pursued. In particular, a  more intrinsic, argument  might help answering the following natural problem.

\begin{problem}
Can one replace $U$, respectively $V$, in the statement of Theorem \ref{T:Rossi} by $\widehat U$ the formal completion of $\mathbb P^n$ along $X$, respectively  a  formal subvariety $\mathscr V \subset \widehat U$ ?
\end{problem}

\subsection{A variant of Rossi's Theorem} The hypothesis of Theorem \ref{T:Rossi} are never satisfied when the ambient is $\mathbb P^3$. The variant of it below, in contrast, can also be applied for analytic subvarieties of $\mathbb P^3$.

\begin{thm}\label{T:Rossibis}
Let $X$ be an irreducible subvariety of $\mathbb P^n$. Let $U$ be an Euclidean neighborhood of $X$ and let $V_1$ and $V_2$  be local irreducible subvarieties of $U$. If $V_1 \cap V_2 = X$ and $\dim V_1 + \dim V_2 > n$
then there exists a subvarieties $\overline V_1$ and $\overline V_2$ of $\mathbb P^n$ such that $\dim \overline V_i = \dim V_i$ and $\overline V_i \cap U \supseteq V_i$ for $i =1 ,2$.
\end{thm}
\begin{proof}
The proof follows the same lines of the proof of Theorem \ref{T:Rossi}. The only difference is that we now
consider $g^* V_i \cap V_j$, for $g \in \Aut(\mathbb P^n)$ sufficiently  close to the identity, in order to produce many positive dimensional projective subvarieties contained in $V_j$.
\end{proof}

\section{Algebraic separatrices for non-dicritical foliations}\label{S:proofs}
\subsection{Existence of algebraic separatrices (Proof of Theorem \ref{THM:Algebraic})}
We first recall that the singular set of any codimension one foliation on $\mathbb P^n$, $n \ge 2$, has an irreducible component of codimension $2$, see for instance \cite[Proposition 2.6, page 95]{MR537038}. Let $S \subset \mathrm{sing}(\F)$ be one such irreducible component, and let $p \in S$ be a sufficiently general point. Let $\Sigma \subset \mathbb P^4$ be a linearly embedded $\mathbb P^2$ passing through $p$. If $\Sigma$ is sufficiently general then $\Sigma$ is not $\F$-invariant and, consequently, the restriction of $\F$ to $\Sigma$ defines a foliation $\mathcal G$ on $\Sigma$.

Theorem \ref{T:CamachoSad} guarantees the existence of a curve $\gamma \subset \Sigma$  everywhere tangent to $\mathcal G$. But then $\gamma$ is clearly every tangent to $\F$ and we can apply Theorem \ref{T:semilocal} to produce an open neighborhood $U$ of $S$ and an analytic subvariety $V \subset U$ containing $S$ which is left  invariant by $\F$. We conclude the proof by applying Theorem \ref{THM:Rossi}. \qed

\subsection{Characterization of non-dicritical logarithmic foliations (Proof of Theorem \ref{THM:Logarithmic})}
As in the proof of Theorem \ref{THM:Algebraic}, let $\Sigma$ be a sufficiently general $\mathbb P^2$ linearly embedded in $\mathbb P^4$ and let $\mathcal G$ be the restriction of $\mathcal F$ to $\Sigma$.
According to \cite[Lemma 10]{MR1185124}, the singular set of $\G$ coincides with the union of $\sing(\F)\cap \Sigma$ and finitely many extra singularities.

 The non-dicriticalness of  $\F$, combined with Darboux-Jouanolou Theorem \cite[Theorem 3.3, page 102]{MR537038} (see also \cite{MR1753461}), implies that $\F$ has finitely many invariant algebraic hypersurfaces. Since $\Sigma$ is general, we can assume that it intersects the smooth locus of the invariant algebraic hypersurfaces transversely. Therefore the separatrices through $p \in \sing(\G) - \sing(\F)\cap \Sigma$ are not contained in $\F$-invariant algebraic hypersurfaces.

 Let $\pi: \tilde \Sigma \to \Sigma$ be the composition of a reduction of singularities of $\G$ with one extra blow-up at every  point in $\sing(\F) \cap \Sigma$  and let $\tilde \G = \pi^* \G$ be the resulting reduced foliation. The extra blow-ups  guarantee that any irreducible curve invariant by $\tilde \G$ is smooth. Since, by assumption, $\tilde \G$ has no saddle-nodes, every singularity of $\tilde \G$ is at the intersection of two germs of convergent separatrices.

The proof of Theorem \ref{THM:Algebraic} shows that every separatrix of $\G$ through a point $p \in \sing(\F) \cap \Sigma$ is algebraic.
Let $\mathscr C= \{ C_1, \ldots, C_k \}$  be the collection of all irreducible $\tilde \G$-invariant algebraic curves which are either an irreducible component of the exceptional divisor or are irreducible components of strict transform of the intersection of an $\F$-invariant algebraic hypersurface with $\Sigma$. Notice that the singular set $\sing(\tilde \G)\cap C_i$ coincides with the intersection of $C_i$ with the divisor $\sum_{i\neq j} C_j$.

Recall from \cite[Chapter 2, Proposition 3]{MR2114696} the formula
\[
    N_{\tilde \G} \cdot C_i = C_i^2 + Z(\tilde \G, C_i) \, .
\]
In our situation, we can write $Z(\tilde \G, C_i)  = (\sum_{j\neq i} C_j ) \cdot C_i$ and deduce that
\begin{equation}\label{E:unica}
    N_{\tilde \G} \cdot C_i = (\sum_{j=1}^ k C_j) \cdot C_i \, .
\end{equation}
for every curve $C_i$, $i =1, \ldots, k$.

The Picard group of $\tilde \Sigma$ is generated by the line bundles associated to the exceptional divisors and by the pull-back of $\mathcal O_{\mathbb P^2}(1)$. Therefore the curves in $\mathscr C$ generate a finite index subgroup of the Picard group of $\tilde \Sigma$.  We deduce that $N_{\tilde \G}$ and $\mathcal O_{\tilde \Sigma}(\sum_{j=1}^ k C_j)$ are isomorphic line bundles because they have equal intersection numbers with every curve in $\mathscr C$ according to (\ref{E:unica}) and the intersection form is non-degenerate in $\Pic(\tilde \Sigma)$.  It follows that $\tilde \G$ is defined by a logarithmic $1$-form with poles on the simple normal crossing divisor $\sum_{i=1}^k C_i$. A result by Deligne \cite[(3.2.14)]{MR0498551} guarantees that such logarithmic $1$-form is closed.  Consequently, the  foliation $\G$ is  defined by a closed logarithmic $1$-form. Since $\Sigma$ is generic,  it follows that   $\F$ is also defined by a closed logarithmic $1$-form, see \cite{MR704017}, \cite[Lemma 9]{MR1185124} or \cite[Appendix A]{2015arXiv151206623C}.  \qed

\subsection{Results for foliations on $\mathbb P^3$}\label{S:dim3}
One can try to replace  Theorem \ref{T:Rossi} by   Theorem \ref{T:Rossibis} in order to establish analogues of Theorems \ref{THM:Algebraic} and \ref{THM:Logarithmic} for foliations on $\mathbb P^3$. The only obstruction in order to do so is the presence of codimension two components of the singular set contained in exactly one semi-local separatrix. This may happen because over a general transverse section the germ of foliation has  a singularity with only one irreducible separatrix (e.g. Poincar\'e-Dulac singularities), or  due to a transitive monodromy action on the set of the separatrices on a two-dimensional section.

The argument used to prove Theorem \ref{THM:Algebraic} implies the
following result.

\begin{prop}
Let $\F$ be a codimension one foliation on $\mathbb P^3$. If $\F$
is non-dicritical then $\F$ leaves invariant an algebraic hypersurface, or
each irreducible component of the singular set of dimension two is contained in
exactly one convergent separatrix.
\end{prop}

Similarly, the arguments leading to Theorem \ref{THM:Logarithmic} also lead to the following result.

\begin{prop}
Let $\F$ be a codimension one foliation on $\mathbb P^3$. Assume that $\F$
is non-dicritical and that  each one-dimensional irreducible component of its
singular set  is contained in at least two convergent separatrices.  If
the restriction of $\F$ to a general $\mathbb P^2$ does not have saddle nodes in its resolution of singularities then $\F$ is defined by a closed logarithmic $1$-form.
\end{prop}

The proposition above gives an alternative proof of the stability of logarithmic foliations
on $\mathbb P^n$, $n\ge 3$, a result established by Calvo-Andrade in a more general context, see \cite{MR1286897}. It suffices to observe that  deformations of general logarithmic foliations are non-dicritical and have one-dimensional irreducible components of the singular set contained in exactly two  separatrices.

\bibliographystyle{plain}

\end{document}